\documentclass{article}
\usepackage{amsmath,amsthm,amsfonts}
\usepackage[utf8]{inputenc}

\newcommand{\lessp}{\prec}

\newcommand{\dual}[1]{\overline{#1}}

\newcommand{\poch}[3]{\left(#1;#2\right)_{#3}}
\newcommand{\pochq}[2]{\left(#1\right)_{#2}}
\newcommand{\pochqn}[1]{\pochq{#1}{n}}
\newcommand{\cO}{O} 

\newcommand{\cS}{\mathcal{S}}

\newcommand{\Fi}{F_{\text{1}}}
\newcommand{\Fii}{F_{\text{2}}}
\newcommand{\Fiii}{F_{\text{3}}}

\newcommand{\Gi}{G_{\text{1}}}
\newcommand{\Gii}{G_{\text{2}}}
\newcommand{\Giii}{G_{\text{3}}}

\newtheorem{theorem}{Theorem}[section]
\newtheorem*{theorem*}{Theorem}
\newtheorem{corollary}[theorem]{Corollary}
\newtheorem*{corollary*}{Corollary}
\newtheorem{lemma}[theorem]{Lemma}
\newtheorem*{lemma*}{Lemma}
\newtheorem{proposition}[theorem]{Proposition}
\newtheorem*{proposition*}{Proposition}
\newtheorem{conjecture}[theorem]{Conjecture}
\newtheorem*{conjecture*}{Conjecture}
\newtheorem{fact}[theorem]{Fact}
\theoremstyle{definition}

\newtheorem*{definition*}{Definition}

\newtheorem*{example*}{Example}
\newtheorem{problem}[theorem]{Problem}
\newtheorem*{problem*}{Problem}

\theoremstyle{remark}

\title{On $q$-Series Identities Related to Interval Orders}
\date{}

\author{George E. Andrews\thanks{Mathematics Department, Pennsylvania State
University,  University Park, PA, USA, \texttt{andrews@math.psu.edu}. Partially
supported by NSA grant H98230-12-1-0205.}
\and V\'\i t Jelínek\thanks{Computer
Science
Institute, Charles University, Prague, \texttt{jelinek@iuuk.mff.cuni.cz}.
Supported by project CE-ITI (GBP202/12/G061) of the Czech Science
Foundation.}
}

\begin{document}

\maketitle
\begin{abstract}
We prove several power series identities involving the refined generating
function of interval orders, as well as the refined generating function of the
self-dual interval orders. These identities may be expressed as 
\[
 \sum_{n\ge 0}\poch{\frac{1}{p}}{\frac{1}{q}}{\!\!n}=
\sum_{n\ge 0} pq^n\poch{p}{q}{n}\poch{q}{q}{n}
\]
and
\[
  \sum_{n\ge 0} (-1)^n\poch{\frac{1}{p}}{\frac{1}{q}}{\!\!n}=
\sum_{n\ge 0} pq^n\poch{p}{q}{n}\poch{-q}{q}{n} 
= \sum_{n\ge0} \left(\frac{q}{p}\right)^n\poch{p}{q^2}{n},
\]
where the equalities apply to the (purely formal) power series expansions of the
above expressions at $p=q=1$, as well as at other suitable roots of unity. 
\end{abstract}

\section{Introduction and Combinatorial Motivation}

Throughout this paper, we use the notation $\poch{a}{q}{n}$ for the
$q$-Pochhammer symbol, defined as
\[
 \poch{a}{q}{n}=\prod_{k=0}^{n-1} (1-aq^k),
\]
with $\poch{a}{q}{0}=1$. Where $q$ is understood from the context, we
write $\pochq{a}{n}$ instead of $\poch{a}{q}{n}$ for brevity. 

The main goal of this paper is to prove new identities for the generating
functions of interval orders and self-dual interval orders. The identities we
derive may be stated as
\[
 \sum_{n\ge 0}\poch{\frac{1}{p}}{\frac{1}{q}}{\!\!n}=
\sum_{n\ge 0} pq^n\poch{p}{q}{n}\poch{q}{q}{n}
\]
and
\[
  \sum_{n\ge 0} (-1)^n\poch{\frac{1}{p}}{\frac{1}{q}}{\!\!n}=
\sum_{n\ge 0} pq^n\poch{p}{q}{n}\poch{-q}{q}{n} 
= \sum_{n\ge0} \left(\frac{q}{p}\right)^n\poch{p}{q^2}{n},
\]
where the equalities mean that the corresponding expressions admit the same
power series expansion as $p$ and $q$ approach 1. It follows from our argument
that the equalities are in fact valid when $p$ and $q$ approach other roots of
unity as well, provided the corresponding series expansions exist.

\subsection{Interval Orders}
Let $P$ be a poset with a strict order relation $\lessp$. We say that $P$ is an
\emph{interval order} if we can assign to each element $x\in P$ a closed real
interval $[l_x,r_x]$  in such a way that $x\lessp y$ if and only if $r_x<l_y$.
As shown by Fishburn~\cite{FishburnPrvni}, a poset is an interval order if and
only if it does not contain a subposet isomorphic to the disjoint union of two
chains of size two. In this paper, we are interested in \emph{unlabelled}
interval orders, i.e., we treat isomorphic posets as identical.

A \emph{Fishburn matrix} is an upper-triangular square matrix of nonnegative
integers with the property that every row and every column has at least one
nonzero entry. The \emph{size} of a matrix is defined as the sum of its entries.
As implied in the work of Fishburn~\cite{FishburnLength,fishburn1985interval},
there is a bijective correspondence between unlabelled interval orders of $n$
elements and Fishburn matrices of size~$n$. In fact, as pointed out
in~\cite{DJK}, there is a bijection that maps interval orders with $n$ elements
having $r$ minimal and $s$ maximal elements to Fishburn matrices of size $n$,
whose first row sums to $r$ and whose last column sums to~$s$.

Let $f_n$ be the number of unlabelled interval orders on $n$ elements.
The sequence $(f_n)_{n\ge 0}$ is known as the \emph{Fishburn
numbers}~\cite[sequence A022493]{oeis}. Apart from counting interval orders and
Fishburn matrices, the numbers $f_n$ have several other combinatorial
interpretations. For instance, $f_n$ is the number of Stoimenow diagrams with
$n$ arcs~\cite{Stoimenow,Zagier}, the number of ascent sequences of length
$n$~\cite{BMCDK}, the number of certain pattern-avoiding permutations of
order~$n$~\cite{BMCDK,bivincular}, or the number of certain pattern-avoiding
insertion tables~\cite{Levande}.

Zagier~\cite{Zagier} has shown that the generating function of Fishburn numbers
may be expressed as
\[
\sum_{n\ge0} f_n x^n =\sum_{n\ge 0} \prod_{k=1}^n (1-(1-x)^k)=\sum_{n\ge
0}\poch{1-x}{1-x}{n},
\]
and deduced the asymptotics
\[
 f_n= n!
\left(\frac{6}{\pi^2}\right)^{\!n}\!\sqrt{n}\left(\alpha+\cO\left(\frac{1}{n}
\right)\right),\quad \text{
  with } \alpha=\frac{12\sqrt{3}}{\pi^{5/2}}e^{\pi^2/12}.
\]
Subsequently, several authors have obtained refinements of this generating
function, enumerating interval orders with respect to various natural
statistics, such as the number of minimal and maximal
elements~\cite{selfdual,KitaevRemmel}, the number of indistinguishable
elements~\cite{indistin} or the number of distinct
down-sets~\cite{BMCDK,selfdual}. 

In this paper, we focus on the refined enumeration of interval orders by the
number of maximal elements. Let $f_{m,\ell}$ be the number of interval orders of
size $m$ with $\ell$ maximal elements. Recall that $f_{m,\ell}$ is also equal to
the number of Fishburn matrices of size $m$ whose last column sums to~$\ell$.
Kitaev and Remmel~\cite{KitaevRemmel} have shown that 
\begin{align}
 \sum_{m\ge 0, \ell\ge 0} f_{m,\ell} x^{m-\ell}y^\ell
&=1+\sum_{n\ge 0}
\frac{y}{(1-y)^{n+1}}\prod_{k=1}^n\left(1-(1-x)^k\right)\notag\\
&=1+\sum_{n\ge 0}
\left(\frac{1}{(1-y)^{n+1}}-\frac{1}{(1-y)^{n}}\right)\prod_{k=1}
^n\left(1-(1-x)^k\right)\notag\\
&=\sum_{n\ge 0}
\left(\frac{1-x}{1-y}\right)^{n+1}\prod_{k=1}^n\left(1-(1-x)^k\right)\notag\\
&=\sum_{n\ge 0}
\left(\frac{1-x}{1-y}\right)^{n+1}\poch{1-x}{1-x}{n}.\label{eq-f3}
\end{align}
Zagier~\cite{Zagier}, Yan~\cite{Yan} and Levande~\cite{Levande} have
independently obtained another formula for the same generating function,
which has also been conjectured by Kitaev and Remmel~\cite{KitaevRemmel},
namely
\begin{equation}
\sum_{m\ge 0, \ell\ge 0} f_{m,\ell} x^{m-\ell}y^\ell=\sum_{n\ge
0}\poch{1-y}{1-x}{n}.
\label{eq-f1}
\end{equation}
We remark that Jel\'\i nek~\cite{selfdual} has derived a formula for the
generating function counting interval orders by their size and the number of
minimal and maximal elements, which simultaneously generalizes
both~\eqref{eq-f3} and~\eqref{eq-f1}, using the fact that the number of minimal
elements has the same distribution as the number of maximal elements.

\subsection{Self-Dual Interval Orders}
Many families of objects enumerated by the Fishburn numbers admit a natural
involutive symmetry map, which transforms an object into its `mirror image'.
In most cases, the known bijections between Fishburn-enumerated families
commute with the corresponding symmetry maps. This suggests that the mirror
symmetry is an inherent property of Fishburn families, and leads to a natural
problem of enumerating symmetric Fishburn objects, i.e., objects that are fixed
by the symmetry map.

For interval orders, the symmetry map corresponds to poset duality. For a
poset $P$ with a strict order relation $\lessp$, its \emph{dual poset} $\dual P$
has the same elements as $P$ and its order relation $\dual\lessp$ is defined by
$x\dual\lessp y\iff y\lessp x$. Clearly, the dual of an interval order is again
an interval order. A poset is \emph{self-dual} if it is isomorphic to its dual. 

In the above-mentioned correspondence between interval orders and Fishburn
matrices, the self-dual interval orders correspond to Fishburn matrices that are
symmetric with respect to the north-east diagonal. We refer to such matrices as
\emph{self-dual} Fishburn matrices. Formally, an $n\times n$ matrix
$M=(M_{ij})_{i,j=1}^n$ is \emph{self-dual} if it satisfies
$M_{i,j}=M_{n-j+1,n-i+1}$ for all $i$ and $j$. Of course, such a matrix $M$ is
uniquely determined by the entries that lie on or below the north-east diagonal,
i.e., by the entries $M_{i,j}$ with $i+j\ge n+1$; we refer to these entries as
\emph{the south-east entries} of~$M$. Moreover, the entries lying on the
north-east diagonal (i.e., the entries $M_{i,j}$ with $i+j=n+1$) will be called
\emph{the diagonal entries} of~$M$.

The \emph{reduced size} of a matrix $M$ is defined as the sum of its south-east
entries. For the purposes of enumerating self-dual Fishburn matrices, and thus
also self-dual interval orders, the notion of reduced size seems to be more
natural than the notion of size. Let $\cS_m$ be the set of the self-dual
Fishburn matrices of reduced size $m$, and let $\cS_{m,\ell}$ be the set of
those matrices in $\cS_m$ whose last column has sum~$\ell$. Let $s_m$ and
$s_{m,\ell}$ be the cardinalities of $\cS_m$ and $\cS_{m,\ell}$, respectively.

The following two facts were proved by Jelínek~\cite{selfdual} by means of
generating functions, and a bijective proof was subsequently found by Yan and
Xu~\cite{yanxu}.
\begin{fact}
For every $m\ge 1$ and every $\ell$, the set $\cS_{m,\ell}$ contains precisely
$s_{m,\ell}/2$ matrices whose diagonal entries are all zero, and therefore also
$s_{m,\ell}/2$ matrices with at least one nonzero diagonal entry.
\end{fact}

\begin{fact}
Let us call a matrix $M$ a \emph{row-Fishburn matrix} if $M$ is an
upper-triangular matrix of nonnegative integers such that every row has at least
one positive entry. Let $r_{m,\ell}$ be the number of row-Fishburn matrices with
(non-reduced) size $m$ and the sum of the last column equal to $\ell$. For $m\ge
1$ and any $\ell$, we have $r_{m,\ell}=s_{m,\ell}/2$.
\end{fact}

This shows that enumerating self-dual Fishburn matrices by their reduced size is
essentially equivalent to enumerating row-Fishburn matrices by their size. It
is not hard to observe (see~\cite[Theorem~4.1]{selfdual}) that the generating
function of $r_{m,\ell}$ may be expressed as
\begin{align}
\sum_{m,\ell\ge 0} 
r_{m,\ell}x^{m-\ell}y^{\ell}&=
\sum_{n\ge 0} \prod_{k=1}^n\left(\frac{1}{(1-y)(1-x)^{k-1}}-1\right)\notag\\
&=\sum_{n\ge 0}(-1)^n\poch{\frac{1}{1-y}}{\frac{1}{1-x}}{\!n}.
\label{eq-rxy}
\end{align}

Denoting by $r_m$ the number of row-Fishburn matrices of size $m$, we then get
\begin{equation}
 \sum_{m\ge 0}
r_mx^m=\sum_{n\ge 0}\prod_{k=1}^n\left(\frac{1}{(1-x)^{k}}-1\right)
=\sum_{n\ge 0}(-1)^n\poch{\frac{1}{1-x}}{\frac{1}{1-x}}{\!n}.
\label{eq-r}
\end{equation}

The sequence $(r_m)_{m\ge 0}$ is listed as A158691 in the OEIS~\cite{oeis}.
Peter Bala, who is the author of the OEIS entry, has pointed out that apparently
the same coefficient sequence arises from expanding a different expression,
namely
\begin{equation}
 \sum_{n\ge 0}\prod_{k=1}^n\left(1-(1-x)^{2k-1}\right)= \sum_{n\ge 0}
\poch{1-x}{(1-x)^2}{n}.
\label{eq-r2}
\end{equation}
He conjectured that \eqref{eq-r} and \eqref{eq-r2} indeed determine the same
power series.

In this paper, we prove the identity conjectured by Bala. We actually extend
this identity to the bivariate generating function of $r_{m,\ell}$ from
\eqref{eq-rxy}, and moreover, we derive yet another, third way of expressing
this generating function. Apart from that, we derive similar identities for the
generating function of $f_{m,\ell}$, providing a third expression for this
generating function, different from those given in \eqref{eq-f3} and
\eqref{eq-f1}. 

It is remarkable that the identities involving the generating function of
$r_{m,\ell}$ turn out to be analogous to those involving the generating
function of $f_{m,\ell}$. In fact, in some cases the identities for the
two generating functions may be deduced from the same general rule by a
different choice of a parameter.

In the course of preparation of our manuscript, we have been informed that
Bringmann, Li and Rhoades~\cite{bring} have independently obtained another proof
of Bala's conjecture, as well as several other identities involving the
generating functions of Fishburn and row-Fishburn matrices. Most, but not all,
of the identities derived by Bringmann, Li and Rhoades also follow from our
Theorem~\ref{thm-main} by setting $y$ equal to~$x$. Apart from the power series
identities, Bringmann, Li and Rhoades have obtained an asymptotic estimate for
the number of row-Fishburn matrices.

\begin{theorem}[Bringmann, Li, Rhoades~\cite{bring}] Let $r_m$ be the number of
row-Fishburn matrices of size $m$. Then, as $m\to\infty$, we have 
\[
r_m= m!
\left(\frac{12}{\pi^2}\right)^{\!m}\left(\beta+\cO\left(\frac{1}{m}
\right)\right), \quad \text{
  with } \beta=\frac{6\sqrt{2}}{\pi^{2}}e^{\pi^2/24}.
\]
\end{theorem}

\section{The Results}

Let us define six formal power series as follows:
\begin{align*}
\Fi(x,y)&=\sum_{n\ge 0} \poch{1-y}{1-x}{n},\\
\Fii(x,y)&=\sum_{n\ge 0}
\frac{1}{(1-y)(1-x)^n}\poch{\frac{1}{1-y}}{\frac{1}{1-x}}{\!n}
\poch{\frac{1}{1-x}}{\frac{1}{1-x}}{\!n}, \\
\Fiii(x,y)&=\sum_{n\ge 0}\left(\frac{1-x}{1-y}\right)^{n+1}
\poch{1-x}{1-x}{n},\\
\Gi(x,y)&=\sum_{n\ge 0}(-1)^n\poch{\frac{1}{1-y}}{\frac{1}{1-x}}{\!n},\\
\Gii(x,y)&=\sum_{n\ge 0}(1-y)(1-x)^n\poch{1-y}{1-x}{n}\poch{-(1-x)}{1-x}{n},
\text{ and}\\
\Giii(x,y)&=\sum_{n\ge 0}\left(\frac{1-x}{1-y}\right)^n\poch{1-y}{(1-x)^2}{n}\!.
\end{align*}
It is not hard to see that all the six infinite sums in these definitions are
convergent in the ring of formal power series in $x$ and $y$. For instance, to
see that the sum in the definition of $\Fi(x,y)$ is convergent, it suffices to
note that each monomial in the expansion of $\poch{1-y}{1-x}{n}$ has degree at
least~$n$.

Note that $\Fi(x,y)$ and $\Fiii(x,y)$ correspond to the two formulas for
the generating function of $f_{m,\ell}$ given in
\eqref{eq-f1} and \eqref{eq-f3},
respectively. In particular, it is known that $\Fi(x,y)=\Fiii(x,y)$.
Note also that $\Gi(x,y)$ is the generating function of $r_{m,\ell}$ given
in~\eqref{eq-rxy}, and that $\Giii(x,x)$ is Bala's formula~\eqref{eq-r2}.
In particular, Bala's conjecture corresponds to the
identity $\Gi(x,x)=\Giii(x,x)$.

The next theorem is our main result.

\begin{theorem}\label{thm-main} In the ring of formal power series in $x$ and
$y$, we have the identities
\begin{equation}
 \Fi(x,y)=\Fii(x,y)=\Fiii(x,y),\label{eq-thm1}
\end{equation}
and
\begin{equation}
 \Gi(x,y)=\Gii(x,y)=\Giii(x,y).\label{eq-thm2}
\end{equation}
\end{theorem}
As we pointed out in the introduction, the equality $\Fi(x,y)=\Fiii(x,y)$ has
been previously proven by the combined results of Kitaev and
Remmel~\cite{KitaevRemmel}, Levande~\cite{Levande}, Yan~\cite{Yan} and
Zagier~\cite{Zagier}. We decided to include $\Fiii$ in the statement of
Theorem~\ref{thm-main} anyway, for comparison with the identities involving
the~$G_i$'s.

Note that $\Fi(x,x)$ is obviously equal to $\Fiii(x,x)$, but the remaining
identities of Theorem~\ref{thm-main} remain non-trivial even when restricted to
the case of $x=y$.

By setting $p=1/(1-y)$ and $q=1/(1-x)$ in \eqref{eq-thm1}, and $p=1-y$ and
$q=1-x$ in \eqref{eq-thm2}, the identities of Theorem~\ref{thm-main} can be
expressed concisely as
\begin{align}
\sum_{n\ge 0} \poch{\frac{1}{p}}{\frac{1}{q}}{\!\!n}=
\sum_{n\ge 0} pq^n\poch{p}{q}{n}\poch{q}{q}{n} 
= \sum_{n\ge0}
\left(\frac{p}{q}\right)^{n+1}\poch{\frac{1}{q}}{\frac{1}{q}}{\!\!n}, \text{ and}
\label{eq-comp1}\\
 \sum_{n\ge 0} (-1)^n\poch{\frac{1}{p}}{\frac{1}{q}}{\!\!n}=
\sum_{n\ge 0} pq^n\poch{p}{q}{n}\poch{-q}{q}{n} 
= \sum_{n\ge0} \left(\frac{q}{p}\right)^n\poch{p}{q^2}{n}. 
\label{eq-comp2}
\end{align}
Note, however, that the expressions in~\eqref{eq-comp1} and \eqref{eq-comp2} are
in general not power series in $p$ and $q$; they should instead be understood as
power series in variables $p-1$ and $q-1$ to make the identities meaningful. 

In fact, the identities~\eqref{eq-comp1} and \eqref{eq-comp2} can be
interpreted in a broader way. If $p$ and $q$ are complex values such that
$pq^{2k}=1$ for some integer $k$, then all the three summations 
in~\eqref{eq-comp2} involve only finitely many nonzero summands, and the sums
are therefore well defined. A straightforward adaptation of our proof of
Theorem~\ref{thm-main} then shows that the values of the three sums are equal.
Moreover, if we consider complex values $p_0$ and $q_0$ such that
$p_0q_0^{2k}=1$ for infinitely many integers $k$, then one may easily check that
all the three summations in \eqref{eq-comp2} are convergent as power series in
$(p-p_0)$ and $(q-q_0)$. An appropriate adaptation of our proof then shows that
the three power series are equal.

With the identities in~\eqref{eq-comp1}, we need to be more careful. The
left equality is again valid for those values of $p$ and $q$ for which the sums
are both terminating, i.e., for values that satisfy $pq^{k}=1$ for an
integer~$k$. And moreover, if $p_0$ and $q_0$ satisfy $p_0q_0^k=1$ for
infinitely many integers $k$, the two expressions are equal as power series in
$(p-p_0)$ and $(q-q_0)$. This may again be proven by a straightforward
modification of our proof.

On the other hand, it is not clear whether the identities involving the
right-hand side of ~\eqref{eq-comp1}, i.e., the  expression
$\Fiii(1-q,1-p)=\sum_{n\ge0} \left(p/q\right)^{n+1}\poch{1/q}{1/q}{n}$, can also
be extended to other complex values of $p$ and~$q$. Since the equality of
$\Fi(x,y)$ and $\Fiii(x,y)$ is based on the combinatorial interpretation of the
two expressions as generating functions, the proof is only applicable to
expansions in powers of $(p-1)$ and $(q-1)$. However, we conjecture that even
this last equality can be extended to those values where the two sides are
defined (see Conjecture~\ref{conj-f3} for a precise statement). 

We remark that the expression $\sum_{n\ge 0}
\left(p/q\right)^{n+1}\poch{1/q}{1/q}{n}$ on the right-hand side
of~\eqref{eq-comp1} also admits a combinatorially
meaningful expansion into powers of $p$ and $1/q$. More precisely, it is not
hard to see that the expression equals $\sum_{r,s\ge 1} a_{r,s} p^r q^{-s}$
where $a_{r,s}$ is the difference between the number of partitions of $s$
into an odd number of parts and the number of partitions of $s$ into an even
number of parts, where we only consider partitions into distinct parts whose
largest part is~$r$. In particular, for $p=1$ we get the following well known
identities (see e.g. Corollary 1.7 on p. 11, and Ex. 10 on p. 29
in~\cite{partitions}):
\[
 \sum_{n\ge 0} \frac{1}{q^{n+1}} \poch{\frac{1}{q}}{\frac{1}{q}}{\!n}=1-\prod_{n\ge 1}(1-{q^{-n}})=
1-\sum_{n=-\infty}^{\infty} (-1)^n q^{-n(3n-1)/2},
\]
where the second identity is a version of the classical Pentagonal Number
Theorem of Euler.

\subsection{Proof of Theorem~\ref{thm-main}}\label{sec-proof}
The proof of Theorem~\ref{thm-main} is based on the following identity, which
has been discovered by Rogers~\cite{rogers} and independently by
Fine~\cite[eq. (14.1)]{fine}.

\begin{theorem}[Rogers--Fine Identity] For $a$, $b$, $q$ and $t$ such that
$|q|<1$, $|t|<1$ and $b$ is not a negative power of $q$, we have 
 \begin{equation}
  \sum_{n\ge 0} \frac{\pochq{aq}{n}}{\pochq{bq}{n}}t^n=
\sum_{n\ge 0}\frac{\pochq{aq}{n}\pochq{\frac{atq}{b}}{n}b^n t^n
q^{n^2}(1-atq^{2n+1})}{\pochq{bq}{n}\pochq{t}{n+1}}
\label{eq-rf}
 \end{equation}
\end{theorem}

We now show how Theorem~\ref{thm-main} follows from the Rogers--Fine Identity.
Since $\Fi$ and $\Fiii$ are already known to be equal, we only need
to show that $\Fi$ equals $\Fii$, and that $\Gi$, $\Gii$ and $\Gii$ are all
equal. As a first step, we derive a general power series identity which  directly
implies both $\Fi=\Fii$ and $\Gi=\Gii$. 

\begin{proposition}\label{pro-12} For any $r$, we have the following identity
of formal power series in $x$ and $y$:
\begin{equation}
 \sum_{n\ge 0} \poch{\frac{1}{1-y}}{\frac{1}{1-x}}{\!n}r^n=\sum_{n\ge 0}
(1-y)(1-x)^n\poch{1-y}{1-x}{n}\poch{r(1-x)}{1-x}{n}.\label{eq-12}
\end{equation}
\end{proposition}
\begin{proof}
Let us substitute $a=\frac{1-y}{1-x}$, $b=\frac{1-y}{(1-x)z}$, $t=r/z$ and
$q=1-x$ into the Rogers--Fine identity, to obtain
\begin{multline}
 \sum_{n\ge 0} \frac{\pochq{1-y}{n}}{z^n\pochq{\frac{1-y}{z}}{n}}r^n=\\
\sum_{n\ge 0} \frac{\pochq{1-y}{n}\pochq{r(1-x)}{n}
r^n(1-y)^n (1-x)^{n^2-n}\left(z-r(1-y)(1-x)^{2n}\right)}
{z^{2n+1}\pochq{\frac{1-y}{z}}{n}\pochq{\frac{r}{z}}{n+1}}.
\label{eq-sub}
\end{multline}

Let $L(x,y,z)$ and $R(x,y,z)$ denote respectively the left-hand side and the
right-hand side of~\eqref{eq-sub}. Let us verify that both $L(x,y,z)$ and
$R(x,y,z)$ are well defined as formal power series in $x$, $y$ and $z$. To see
that $L(x,y,z)$ is well defined, we first note that the denominator
$z^n\pochq{(1-y)/z}{n}$ of the $n$-th summand on the left-hand side of
\eqref{eq-sub} is a polynomial in $x$, $y$ and $z$ with nonzero constant term,
showing that each summand can be expanded into a power series. It remains to
verify that the sum on the left-hand side of $\eqref{eq-sub}$ is convergent in
the ring of formal power series. This follows from the fact that every monomial
$x^iy^j$ appearing with nonzero coefficient in the expansion of
$\poch{1-y}{1-x}{n}$ satisfies $i+j\ge n$, and therefore a monomial $x^iy^jz^k$
may only appear with nonzero coefficient in the first $i+j$ summands of
$L(x,y,z)$. Thus, $L(x,y,z)$ is well defined. The same reasoning applies to
$R(x,y,z)$ as well.

We now set $z=0$ in $L$ and $R$, to obtain
\begin{align*}
 L(x,y,0)&=\sum_{n\ge0}\frac{\pochq{1-y}{n}}{(-1)^n(1-y)^n(1-x)^{\binom{n}{2}}}
r^n\\
&=\sum_{n\ge0}r^n(-1)^n\prod_{k=1}^n\left(\frac{1}{(1-y)(1-x)^{k-1}}-1\right) \\
&=\sum_{n\ge0}\poch{\frac{1}{1-y}}{\frac{1}{1-x}}{\!n}r^n
\intertext{and}
 R(x,y,0)&=\sum_{n\ge
0}\frac{-\pochq{1-y}{n}\pochq{r(1-x)}{n}r^{n+1}(1-y)^{n+1} (1-x)^{n^2+n}}
{(-1)^{2n+1}r^{n+1}(1-y)^n(1-x)^{n^2}}\\
&=\sum_{n\ge0}(1-y)(1-x)^n\pochq{1-y}{n}\pochq{r(1-x)}{n},
\end{align*}
as claimed.
\end{proof}

\begin{corollary}\label{cor-12}
 $\Fi(x,y)=\Fii(x,y)$ and $\Gi(x,y)=\Gii(x,y)$.
\end{corollary}
\begin{proof}
Taking $r=-1$ in \eqref{eq-12} shows that $\Gi(x,y)=\Gii(x,y)$. By taking
$r=1$ and substituting $x=-x'/(1-x')$ and $y=-y'/(1-y')$ in~\eqref{eq-12}, we
prove that $\Fi(x',y')=\Fii(x',y')$.
\end{proof}

\begin{lemma}\label{lem-g13}
 $\Gi(x,y)=\Giii(x,y)$.
\end{lemma}
\begin{proof}
We again use the Rogers--Fine identity. This time, we substitute
$a=(1-x)^2/(1-y)$, $b=(1-x)/z$, $q=(1-x)^{-2}$ and $t=1/z$.
This yields
\begin{multline}
 \sum_{n\ge
0}\frac{\pochq{\frac{1}{1-y}}{n}}{z^n\pochq{\frac{1}{z(1-x)}}{n}}\\
 =\sum_{n\ge0}
\frac{\pochq{\frac{1}{1-y}}{n}\pochq{\frac{1}{(1-y)(1-x)}}{n}(1-x)^{n-2n^2}
\left(z-\frac{1}{(1-y)(1-x)^{4n}}\right)}
{z^{2n+1}\pochq{\frac{1}{z(1-x)}}{n}\pochq{\frac{1}{z}}{n+1}}.
\label{eq-sub2}
\end{multline}

Let $L'(x,y,z)$ and $R'(x,y,z)$ denote the left-hand side and right-hand side
of $\eqref{eq-sub2}$, respectively. As with $L$ and $R$ in the proof of
Proposition~\ref{pro-12}, we easily observe
that $L'$ and $R'$ are formal power series in $x$, $y$ and~$z$. Putting $z$
equal to $0$, we get

\begin{align*}
 L'(x,y,0)&=\sum_{n\ge 0}\frac{\pochq{\frac{1}{1-y}}{n}}{
\prod_{k=1}^n-\frac{1}{(1-x)^{2k-1}}}\\
&=\sum_{n\ge 0}\prod_{k=1}^n\left(\frac{1-x}{1-y}-(1-x)^{2k-1}\right)\\
&=\Giii(x,y),
\intertext{and}
R'(x,y,0)&=\sum_{n\ge0}
\frac{-\pochq{\frac{1}{1-y}}{n}\pochq{\frac{1}{(1-y)(1-x)}}{n}\frac{1}{1-y}
(1-x)^{-2n^2-3n}}
{-(1-x)^{-2n^2-n}}\\
&=\sum_{n\ge0}\frac{1}{(1-y)(1-x)^{2n}}\poch{\frac{1}{1-y}}{\frac{1}{1-x}}{\!2n}
\\ 
&=\sum_{n\ge0}\left(\poch{\frac{1}{1-y}}{\frac{1}{1-x}}{\!2n}-\poch{\frac{1}{1-y}}
{\frac{1}{1-x}}{\!2n+1}\right)\\
&=\sum_{n\ge 0}(-1)^n \poch{\frac{1}{1-y}}{\frac{1}{1-x}}{\!n}\\
&=\Gi(x,y).\qedhere
\end{align*}
\end{proof}
Theorem~\ref{thm-main} is a direct consequence of Proposition~\ref{pro-12} and
Lemma~\ref{lem-g13}.

\section{A Generalization}

We are able to prove the following generalization of the Rogers--Fine
identity:

\begin{theorem}[Generalized Rogers--Fine Identity]
\label{thm-grf}
\begin{multline}
 \sum_{n\ge 0} \frac{\pochq{\frac{\beta\gamma}{\alpha qt}}{n}\pochq{\alpha}{n}}
{\pochq{\beta}{n}\pochq{\gamma}{n}} t^n\\
=\sum_{n\ge 0} \frac{\pochq{\frac{\alpha qt}{\beta}}{n}\pochq{\frac{\alpha
qt}{\gamma}}{n}\pochq{\alpha}{n}\left(1-\alpha tq^{2n}\right)(-1)^n
q^{\binom{n}{2}-n}\left(\frac{\beta\gamma}{\alpha}\right)^n}
{\pochq{\beta}{n}\pochq{\gamma}{n}\pochq{t}{n+1}}.
\label{eq-grf}
\end{multline}
\end{theorem}
To deduce the Rogers--Fine identity from Theorem~\ref{thm-grf}, simply take the
limit $\gamma\to 0$ and set $\alpha=aq$ and $\beta=bq$.
\begin{proof}[Proof of Theorem~\ref{thm-grf}]
Our argument is based on the following
identity of Watson (see e.g.~\cite[eq. (3.4.1.5.)]{sla}), valid for $f=q^{-N}$
with $|q|<1$ and $N$ a positive integer:
\begin{multline}
 \sum_{n\ge
0}\frac{\pochqn{a}\pochqn{b}\pochqn{c}\pochqn{d}\pochqn{e}\pochqn{f}(1-aq^{2n}
)(a^2q^2/bcdef)^n}{\pochqn{q}\pochqn{aq/b}\pochqn{aq/c}\pochqn{aq/d}
\pochqn{aq/e}\pochqn{aq/f}(1-a)}\\
=\frac{\pochq{aq}{N}\pochq{aq/de}{N}}{\pochq{aq/d}{N}\pochq{aq/e}{N}}\sum_{n\ge0
}\frac{\pochqn{aq/bc}\pochqn{d}\pochqn{e}\pochqn{f}q^n}{\pochqn{q}\pochqn{def/a}
\pochqn{aq/b}\pochqn{aq/c}}.
\label{eq-wat}
\end{multline}
In \eqref{eq-wat}, we put $d=q$ and take the limit as $N\to\infty$, to get
\begin{multline}
\sum_{n\ge0}\frac{\pochqn{b}\pochqn{c}\pochqn{e}q^{\binom{n}{2}-n}
(1-aq^ { 2n } )(-a^2/bce)^n }
{\pochqn{aq/b}\pochqn{aq/c}\pochqn{aq/e}(1-a)}\\
=\frac{1-a/e}{1-a}\sum_{n\ge0}\frac{\pochqn{aq/bc}\pochqn{e}(a/e)^n}{\pochqn{
aq/b} \pochqn{aq/c}}.
\label{eq-wat2}
\end{multline}
Putting $a=\alpha t$, $b=\alpha qt/\beta$, $c=\alpha qt/\gamma$, and
$e=\alpha$ in \eqref{eq-wat2}, and multiplying the resulting identity by
$(1-\alpha t)/(1-t)$, we obtain~\eqref{eq-grf}.
\end{proof}

From the generalized Rogers--Fine identity, we may deduce a generalization of
Proposition~\ref{pro-12} and Lemma~\ref{lem-g13}, by following the same
arguments we used to prove Proposition~\ref{pro-12} and Lemma~\ref{lem-g13}
from the Rogers--Fine identity. In particular, substituting $\alpha=1-y$,
$\beta=(1-y)/z$, $t=r/z$ and $q=1-x$ into~\eqref{eq-grf}, and taking $z=0$, 
shows that

\begin{multline*}
\sum_{n\ge
0}\frac
{\poch{\frac{\gamma}{r(1-x)}}{1-x}{\!n}\poch{\frac{1}{1-y}}{\frac{1}{1-x}}{\!n}}
{\poch{\gamma}{1-x}{n}}r^n\\
=\sum_{n\ge0}(1-y)(1-x)^n\frac{\poch{1-y}{1-x}{n}\poch{r(1-x)}{1-x}{n}}
{\poch{\gamma}{1-x}{n}}.
\end{multline*}

Similarly, by putting $\alpha=1/(1-y)$, $\beta=1/z(1-x)$,
$t=1/z$, and $q=1/(1-x)^2$, we get
\begin{multline*}
\sum_{n\ge0} (-1)^n
\frac{\poch{\frac{1}{1-y}}{\frac{1}{1-x}}{n}}
{\poch{\gamma}{\frac{1}{(1-x)^2}}{\lfloor \frac{n}{2}\rfloor}}
\\=\sum_{n\ge
0}\left(\frac{1-x}{1-y}\right)^n\frac{\poch{1-y}{(1-x)^2}{n}\poch{
\gamma(1-x)(1-y)}{\frac{1}{(1-x)^2}}{n}}
{\poch{\gamma}{\frac{1}{(1-x)^2}}{n}}.
\end{multline*}

Unfortunately, we are not able to find a combinatorial interpretation for
these generalized identities.

\section{Open Problems}

Let us recall the identities between the function $\Fiii$ and the two functions
$\Fi$ and $\Fii$. With the notation $p=1/(1-y)$ and $q=1/(1-x)$, the identities
may be stated as
\begin{align}
\sum_{n\ge 0} \poch{\frac{1}{p}}{\frac{1}{q}}{\!n}&=
\sum_{n\ge 0}
\left(\frac{p}{q}\right)^{n+1}\poch{\frac{1}{q}}{\frac{1}{q}}{\!n}, \text{ 
and}\label{eq-conj1}\\
\sum_{n\ge 0} pq^n\poch{p}{q}{n}\poch{q}{q}{n}&=
\sum_{n\ge 0}
\left(\frac{p}{q}\right)^{n+1}\poch{\frac{1}{q}}{\frac{1}{q}}{\!n}.
\label{eq-conj2}
\end{align}
By Theorem~\ref{thm-main}, the identities~\eqref{eq-conj1} and \eqref{eq-conj2}
are valid in the ring of power series in $(p-1)$ and $(q-1)$. However, it seems
that the identities are in fact valid for any value of $p$ and $q$ for which the
sums are terminating, and they also appear to be valid as power series in
$(p-p_0)$ and $(q-q_0)$ whenever the corresponding sums both converge as formal
power series. We state this more precisely in the next conjecture.
\begin{conjecture}\label{conj-f3}
If $p_0$ and $q_0$ are two complex $k$-th roots of unity for some $k$, then the
left-hand side and the right-hand side of \eqref{eq-conj1} converge to the same
power series in $(p-p_0)$ and $(q-q_0)$.

Similarly, if $q_0$ is a root of unity, then both sides of \eqref{eq-conj2}
converge to the same power series in $(q-q_0)$.
\end{conjecture}

Finally, we note that although the power series from Theorem~\ref{thm-main} have
a natural combinatorial interpretation as generating functions of combinatorial
objects, our proof of Theorem~\ref{thm-main} does not use this interpretation
at all. One might ask whether there is a way to interpret the identities
of Theorem~\ref{thm-main} combinatorially and thus provide an alternative proof.

\begin{problem} 
Apart from the (previously known) identity $\Fi(x,y)=\Fiii(x,y)$, 
is there a combinatorial proof of the identities in Theorem~\ref{thm-main}?
\end{problem}


\begin{thebibliography}{10}

\bibitem{partitions}
G.~E. Andrews.
\newblock {\em The Theory of Partitions}.
\newblock Cambridge University Press, 1988.

\bibitem{BMCDK}
M.~Bousquet-M\'{e}lou, A.~Claesson, M.~Dukes, and S.~Kitaev.
\newblock (2+2)-free posets, ascent sequences and pattern avoiding
  permutations.
\newblock {\em J. Combin. Theory Ser. A}, 117(7):884--909, 2010.

\bibitem{bring}
K.~Bringmann, Y.~Li, and R.~C. Rhoades.
\newblock Asymptotics for the number of row-{F}ishburn matrices.
\newblock {\em Manuscript}, 2013.

\bibitem{indistin}
M.~Dukes, S.~Kitaev, J.~Remmel, and E.~Steingr\'{\i}msson.
\newblock Enumerating $(2+ 2)$-free posets by indistinguishable elements.
\newblock {\em Journal of Combinatorics}, 2(1):139--163, 2011.

\bibitem{DJK}
M.~Dukes, V.~Jel\'\i nek, and M.~Kubitzke.
\newblock Composition matrices, $(2+2)$-free posets and their specializations.
\newblock {\em Electron. J. Combin.}, 18(1)(P44), 2011.

\bibitem{fine}
N.~J. Fine.
\newblock {\em Basic Hypergeometric Series and Applications}.
\newblock American Mathematical Society, 1988.

\bibitem{FishburnPrvni}
P.~C. Fishburn.
\newblock Intransitive indifference with unequal indifference intervals.
\newblock {\em J. Math. Psychol.}, 7(1):144 -- 149, 1970.

\bibitem{FishburnLength}
P.~C. Fishburn.
\newblock Interval lengths for interval orders: A minimization problem.
\newblock {\em Discrete Math.}, 47:63 -- 82, 1983.

\bibitem{fishburn1985interval}
P.~C. Fishburn.
\newblock {\em Interval orders and interval graphs: A study of partially
  ordered sets}.
\newblock John Wiley \& Sons, 1985.

\bibitem{oeis}
OEIS~Foundation Inc.
\newblock The {O}n-{L}ine {E}ncyclopedia of {I}nteger {S}equences.
\newblock http://oeis.org/, 2011.

\bibitem{selfdual}
V.~Jel{\'i}nek.
\newblock Counting general and self-dual interval orders.
\newblock {\em J. Combin. Theory Ser. A}, 119(3):599--614, 2012.

\bibitem{KitaevRemmel}
S.~Kitaev and J.~Remmel.
\newblock Enumerating (2+2)-free posets by the number of minimal elements and
  other statistics.
\newblock In {\em 22nd International Conference on Formal Power Series and
  Algebraic Combinatorics ({FPSAC} 2010)}, 2010.

\bibitem{Levande}
P.~Levande.
\newblock {F}ishburn diagrams, {F}ishburn numbers and their refined generating
  functions.
\newblock {\em J. Combin. Theory Ser. A}, 120(1):194--217, January 2013.

\bibitem{bivincular}
R.~Parviainen.
\newblock Wilf classification of bi-vincular permutation patterns.
\newblock {\em arXiv:0910.5103}, 2009.

\bibitem{rogers}
L.~J. Rogers.
\newblock On two theorems of combinatory analysis and some allied identities.
\newblock {\em Proc. London Math. Soc.}, 16(2), 1917.

\bibitem{sla}
L.~J. Slater.
\newblock {\em Generalized hypergeometric functions}.
\newblock Cambridge University Press, 1966.

\bibitem{Stoimenow}
A.~Stoimenow.
\newblock Enumeration of chord diagrams and an upper bound for {V}assiliev
  invariants.
\newblock {\em J. Knot Theory Ramifications}, 7:93 -- 114, 1998.

\bibitem{Yan}
S.~H.~F. Yan.
\newblock On a conjecture about enumerating (2+2)-free posets.
\newblock {\em European J. Combin.}, 32(2):282 -- 287, 2011.

\bibitem{yanxu}
S.~H.~F. Yan and Y.~Xu.
\newblock Self-dual interval orders and row-{F}ishburn matrices.
\newblock {\em Electron. J. Combin.}, 19(2):P5, 2012.

\bibitem{Zagier}
D.~Zagier.
\newblock Vassiliev invariants and a strange identity related to the {D}edekind
  eta-function.
\newblock {\em Topology}, 40(5):945--960, 2001.

\end{thebibliography}

\end{document}